\documentclass[11pt]{article}
\pdfoutput=1
\usepackage[T1]{fontenc}
\usepackage[utf8]{inputenc}
\usepackage[russian, french, english]{babel}
\usepackage[margin=0.8in]{geometry}
\usepackage{amsmath,amssymb,amstext}
\usepackage{amsthm}
\usepackage{authblk}
\usepackage{enumerate}
\usepackage{enumitem}
\usepackage{dirtytalk}
\usepackage{tikz}
\usepackage{thmtools}
\usepackage{thm-restate}
\usepackage{multicol}
\usetikzlibrary{shapes}
\usepackage{xspace}
\usepackage{xcolor}
\usepackage{tikz}
\usepackage{thmtools}
\usepackage{thm-restate}
\usetikzlibrary{shapes}
\usepackage{xcolor}
\usepackage[spacing=true,kerning=true,babel=true]{microtype}
\usepackage[justification=centering, labelfont=bf, font=small, width=.75\textwidth]{caption}
\usepackage{framed}
\usepackage{mathtools}
\usepackage{lmodern}
\usepackage{algorithm}
\usepackage{algpseudocode}
\usepackage{multicol}
\usepackage[hidelinks]{hyperref}
\usepackage[backend=biber, style=alphabetic, sorting=nyt, maxnames=100,backref=true,sortcites]{biblatex}
\newtheoremstyle{slantthm}%
{}{}%
{\slshape}{}%
{\bfseries}{\bfseries.}%
{ }%
{\thmname{#1}\thmnumber{ #2}\thmnote{ \ep{\normalfont{}#3}}}

\newtheoremstyle{claim}%
{}{}%
{\slshape}{}%
{\itshape}{.}%
{ }%
{\thmname{#1}\thmnumber{ #2}\thmnote{ \ep{\normalfont{}#3}}}

\theoremstyle{slantthm}
\newtheorem{thm}{Theorem}[section]
\newtheorem{lemma}[thm]{Lemma}

\newtheorem{claim}{Claim}%[]
\newtheorem*{false}{``Lemma''}

\numberwithin{equation}{section}
\theoremstyle{definition} 
\newtheorem{definition}[thm]{Definition}

\tikzset{black/.style={shape=circle,draw=black,fill=black,inner sep=1pt, minimum size=9pt}}
\tikzset{white/.style={shape=circle,draw=black,fill=white,inner sep=1pt, minimum size=9pt}}
\tikzset{invisible/.style={shape=circle,draw=black,fill=black,inner sep=0pt, minimum size=0.1pt}}
\tikzset{fg/.style={shape=circle,draw=black,fill=ForestGreen,inner sep=1pt, minimum size=9pt}}
\tikzset{fuchsia/.style={shape=circle,draw=black,fill=Fuchsia,inner sep=1pt, minimum size=9pt}}
\tikzset{blue/.style={shape=circle,draw=black,fill=BlueViolet,inner sep=0pt, minimum size=9pt}}
\tikzset{sepia/.style={shape=circle,draw=black,fill=Sepia,inner sep=0pt, minimum size=9pt}}
\tikzset{peri/.style={shape=circle,draw=black,fill=Periwinkle,inner sep=0pt, minimum size=9pt}}
\tikzset{gold/.style={shape=circle,draw=black,fill=Goldenrod,inner sep=0pt, minimum size=9pt}}
\tikzset{pink/.style={shape=circle,draw=black,fill=WildStrawberry,inner sep=0pt, minimum size=9pt}}
\tikzset{starnode/.style={inner sep=1pt, minimum size=13pt, star,star points=4,star point ratio=0.5, draw, fill=white}}

\definecolor{asparagus}{rgb}{0.53, 0.66, 0.42}

\definecolor{cr}{rgb}{0.7, 0.11, 0.11}

\definecolor{lav}{rgb}{0.45, 0.31, 0.59}

\date{}

\renewcommand{\phi}{\varphi}
\renewcommand{\theta}{\vartheta}
\renewcommand{\epsilon}{\varepsilon}
\renewcommand{\leq}{\leqslant}
\renewcommand{\geq}{\geqslant}
\newcommand{\defeq}{\coloneqq}

\newcommand{\bemph}[1]{{\normalfont#1}}
\newcommand{\ep}[1]{\bemph{(}#1\bemph{)}}

\newcommand{\dom}{\mathsf{dom}}
\numberwithin{equation}{section}
\newcommand{\N}{\mathbb{N}}
\newcommand{\Z}{\mathbb{Z}}
\newcommand{\0}{\varnothing}
\newcommand{\set}[1]{\{#1\}}
\newcommand{\emphd}[1]{{\fontseries{b}\selectfont\textsf{#1}}}
\newcommand{\de}{\mathsf{d}}
\newcommand{\wde}{\mathsf{wd}}
\newcommand{\del}{\mathsf{Delete}}
\newcommand{\delsave}{\mathsf{DelSave}}
\newcommand{\blank}{\mathsf{blank}}

\newcommand{\bdry}[1]{{\partial #1}}

\newcommand\blfootnote[1]{%
  \begingroup
  \renewcommand\thefootnote{}\footnote{#1}%
  \addtocounter{footnote}{-1}%
  \endgroup
}

\renewbibmacro{in:}{}

\renewbibmacro*{volume+number+eid}{%
	\printfield{volume}%
	\setunit*{\addnbspace}% NEW (optional); there's also \addnbthinspace
	\printfield{number}%
	\setunit{\addcomma\space}%
	\printfield{eid}}
\DeclareFieldFormat[article]{volume}{\textbf{#1}\space}
\DeclareFieldFormat[article]{number}{\mkbibparens{#1}}

\DeclareFieldFormat{journaltitle}{#1,}
\DeclareFieldFormat[thesis]{title}{\mkbibemph{#1}\addperiod}
\DeclareFieldFormat[article, unpublished, thesis]{title}{\mkbibemph{#1},}
\DeclareFieldFormat[book]{title}{\mkbibemph{#1}\addperiod}
\DeclareFieldFormat[unpublished]{howpublished}{#1, }

\DeclareFieldFormat{pages}{#1}

\DeclareFieldFormat[article]{series}{#1\addcomma}

\definecolor{nicelavender}{RGB}{153, 128, 250}

\definecolor{xkcd_fern}{RGB}{99, 169, 80}

\begin{document}

%TITLE, ETC

%\author{Luke Postle\thanks{$^{,\dagger}$We acknowledge the support of the Natural Sciences and Engineering Research Council of Canada (NSERC), $^*$[Discovery Grant No.  2019-04304], $^\dagger$[CGSD3 Grant No. 2020-547516] \\
%\hphantom{|} $ ^{*,\dagger}$Cette recherche a \'{e}t\'{e} financ\'{e}e par le Conseil de recherches en sciences naturelles et en g\'{e}nie du Canada (CRSNG), $^*$[Discovery Grant No.  2019-04304], $^\dagger$[CGSD3 Grant No. 2020-547516]} }
\author{Anton~Bernshteyn$^\star$}
\author{Eugene Lee$^\flat$}
\author{Evelyne Smith-Roberge$^\dagger$}
\affil{$^\star$Department of Mathematics, University of California, Los Angeles \\ \texttt{bernshteyn@math.ucla.edu}}
\affil{$^\flat$Independent Researcher}% \\ \texttt{bleh3@hotmail.com}}

\affil{$^\dagger$School of Mathematics, Georgia Institute of Technology \\ \texttt{esmithroberge3@gatech.edu}}

\title{Weak Degeneracy of Planar Graphs}
\date{}

\maketitle

\blfootnote{Research of the first named author is partially supported by the NSF CAREER grant DMS-2528522.}

%------------------------------
%-------  INTRODUCTION --------
%------------------------------

%\vspace*{-30pt}

\begin{abstract}
    \noindent The weak degeneracy of a graph $G$ is a numerical parameter that was recently introduced by the first two authors with the aim of understanding the power of greedy algorithms for graph coloring. Every $d$-degenerate graph is weakly $d$-degenerate, but the converse is not true in general (for example, all connected $d$-regular graphs except cycles and cliques are weakly $(d-1)$-degenerate). If $G$ is weakly $d$-degenerate, then the list-chromatic number of $G$ is at most $d+1$, and the same upper bound holds for various other parameters such as the DP-chromatic number and the paint number. Here we rectify a mistake in a paper of the first two authors and give a correct proof that planar graphs are weakly $4$-degenerate, strengthening the famous result of Thomassen that planar graphs are $5$-list-colorable. 
\end{abstract}

%\maketitle

%added this for a new section style
%\newsectionstyle
%======================================================================
\section{Introduction}\label{sec:intro}
%======================================================================

    All graphs in this paper are finite and simple. We use $\N \defeq \set{0,1,2,\ldots}$ to denote the set of all nonnegative integers, and for $k \in \N$, we let $[k] \defeq \set{i \in \N \,:\, 1 \leq i \leq k}$. Given $d \in \N$, a graph $G$ is \emphd{$d$-degenerate} if the vertices of $G$ can be ordered so that each vertex is preceded by at most $d$ of its neighbors. The \emphd{degeneracy} of $G$, denoted by $\de(G)$, is the least $d \in \N$ such that $G$ is $d$-degenerate. The following simple greedy algorithm shows that the chromatic number of $G$, $\chi(G)$, is at most $\de(G) + 1$:
    
    \begin{algorithm}[h]%\scriptsize
    \caption{Greedy coloring}\label{alg:greedy}
    \begin{algorithmic}[1]

        \medskip
    
        \Statex \textbf{Input:} An $n$-vertex graph $G$ and an integer $d \in \N$.

        \medskip
        
        \State $L(v) \gets [d+1]$ for all $v \in V(G)$.
        \For{$i \gets 1$ to $n$}
            \State Pick a vertex $u_i \in V(G)$. \label{line:choice} %and $w_i \in N_G(u) \cup \set{\blank}$.
            %\If{$w_i \neq \blank$ or $|L(u_i)| > |L(w_i)|$}
            %    \State Assign to $u_i$ an arbitrary color $c_i \in L(u_i) \setminus L(w_i)$.
            %\Else
                \State Assign to $u_i$ an arbitrary color $c_i \in L(u_i)$.
            %\EndIf
            \State Remove $c_i$ from $L(v)$ for all $v \in N_G(u_i)$.
            \State $G \gets G - u_i$.
        \EndFor
    \end{algorithmic}
\end{algorithm}

    \noindent At the start of the $i$-th iteration of the \textsf{for} loop, the set $L(u_i)$ contains all the colors from $[d+1]$ that have not yet been assigned to any neighbors of $u_i$. Therefore, if the ordering $u_1$, \ldots, $u_n$ witnesses the bound $\de(G) \leq d$, then at the $i$-th iteration of the \textsf{for} loop, the set $L(u_i)$ will be nonempty, and thus the algorithm will successfully generate a proper $(d+1)$-coloring of $G$. %  $u_1$, \ldots, $u$
    
    %Since each $u_i$ is preceded by at most $d$ neighbors and $|L(u_i)| = d + 1$ at the start, %it will have at least one available color
    %when the algorithm reaches line~\ref{step:color} we have $|L(u_i)| \geq 1$. Therefore, the algorithm %, and hence it
    %successfully outputs a proper $(d+1)$-coloring of $G$.
    Unfortunately, the upper bound $\chi(G) \leq \de(G) + 1$ is rarely sharp. For example, every $d$-regular graph $G$ satisfies $\de(G) = d$, but, by a theorem of Brooks, the only connected $d$-regular graphs with $\chi(G) = d + 1$ are cliques and odd cycles \cites{Brooks}[Theorem~5.2.4]{Diestel}. To address this issue, the first two authors considered in \cite{WD} a variant of Algorithm~\ref{alg:greedy} in which a vertex $u_i$ may attempt to ``save'' a color for one of its neighbors, $w_i$. In what follows, we use $\blank$ as a special symbol distinct from every vertex of $G$.

    \begin{algorithm}[h]%\scriptsize
    \caption{Greedy coloring with savings}\label{alg:greedy_save}
    \begin{algorithmic}[1]
        \medskip
    
        \Statex \textbf{Input:} An $n$-vertex graph $G$ and an integer $d \in \N$.

        \medskip
        
        \State $L(v) \gets [d+1]$ for all $v \in V(G)$.
        \For{$i \gets 1$ to $n$}
            \State Pick $u_i \in V(G)$ and $w_i \in N_G(u_i) \cup \set{\blank}$.
            \If{$w_i \neq \blank$ and $|L(u_i)| > |L(w_i)|$}
                \State Assign to $u_i$ an arbitrary color $c_i \in L(u_i) \setminus L(w_i)$. \label{step:save}
            \Else
                \State Assign to $u_i$ an arbitrary color $c_i \in L(u_i)$.
            \EndIf
            \State Remove $c_i$ from $L(v)$ for all $v \in N_G(u_i)$.
            \State $G \gets G - u_i$.
        \EndFor
    \end{algorithmic}
    \end{algorithm}
    
    \noindent The assumption that $|L(u_i)| > |L(w_i)|$ guarantees that in line~\ref{step:save}, the set $L(u_i) \setminus L(w_i)$ is nonempty. As a result, if during the $i$-th iteration of the \textsf{for} loop the algorithm reaches line~\ref{step:save}, then the set $L(w_i)$ does not shrink at this iteration, even though $w_i$ is adjacent to $u_i$. By keeping track of a lower bound on the size of $L(v)$ for every vertex $v$ throughout the execution of Algorithm~\ref{alg:greedy_save}, we arrive at the following definition:

    \begin{definition}[{Weak degeneracy \cite{WD}}]
        Let $G$ be a graph and let $f \colon V(G) \to \N$. Given $u \in V(G)$ and $w \in N_G(u) \cup \set{\blank}$, we let $\delsave(G,f,u,w) \defeq (G - u,f')$, where $f' \colon V(G - u) \to \Z$ is given by
            \[
                f'(v) \,\defeq\, \begin{cases}
                    %f(v) - 1 &\text{if } v \in N_G(u) \text{ and } v \neq w,\\
                    %f(w) &\text{if } v = w \text{ and } f(u) > f(w),\\
                    %f(w) - 1 &\text{if } v = w \text{ and } f(u) \leq f(w),\\
                    %f(v) &\text{otherwise}.
                    f(v) &\text{if } v \notin N_G(u) \text{ or } \big(v = w \text{ and } f(u) > f(w)\big),\\
                    f(v) - 1 &\text{otherwise}.
                \end{cases}
            \]
        An application of the $\delsave$ operation is \emphd{legal} if the resulting function $f'$ is non-negative. For clarity, we write $\del(G,f,u) \defeq \delsave(G,f,u,\blank)$. A graph $G$ is \emphd{weakly $f$-degenerate} if, starting with $(G,f)$, it is possible to remove all vertices from $G$ by a sequence of legal applications of the $\delsave$ operation. The \emphd{weak degeneracy} of $G$, denoted by $\wde(G)$, is the minimum $d \in \N$ such that $G$ is weakly degenerate with respect to the constant $d$ function. If $\wde(G) \leq d$ for some $d \in \N$, we say $G$ is \emphd{weakly $d$-degenerate}.
        
        When $G$ and $f$ are clear from the context, we may simply write $\delsave(u,w)$ for $\delsave(G,f,u,w)$ and $\del(u)$ for $\del(G,f,u)$.
    \end{definition}

    We remark that our notation is slightly different but essentially equivalent to the one in \cite{WD} (there the operations $\del$ and $\delsave$ are defined separately).

    The degeneracy of $G$ is the least $d$ such that starting with the constant $d$ function, it is possible to remove all vertices from $G$ via a sequence of legal applications of the $\del$ operation (i.e., by only using the $\delsave$ operation with $\blank$ as the last argument). It follows that $\wde(G) \leq \de(G)$. On the other hand, the above discussion of Algorithm~\ref{alg:greedy_save} shows that $\chi(G) \leq \wde(G) + 1$ for every graph $G$. Moreover, it was proved in \cite{WD} that $\wde(G) + 1$ is an upper bound on a number of other coloring-related parameters:

    \begin{thm}[{AB--EL \cite[Proposition 1.3]{WD}}]\label{thm:wd+1}
        If $G$ is a graph, then $\wde(G) + 1$ is an upper bound on $\chi(G)$, the chromatic number of $G$; $\chi_\ell(G)$, the list-chromatic number of $G$; $\chi_{DP}(G)$, the DP-chromatic number of $G$; $\chi_{P}(G)$, the paint number of $G$; and $\chi_{DPP}(G)$, the DP-paint number of $G$.
        %\begin{multicols*}{2}
        %\begin{itemize}
        %    \item $\chi(G)$, the chromatic number of $G$,
        %    \item $\chi_\ell(G)$, the list-chromatic number of $G$,
        %    \item $\chi_{DP}(G)$, the DP-chromatic number of $G$,
        %    \item $\chi_{P}(G)$, the paint number of $G$,
        %    \item $\chi_{DPP}(G)$, the DP-paint number of $G$.
        %\end{itemize}
        %\end{multicols*}
    \end{thm}

    Since these parameters will not be directly used in the sequel, we will not define them here and only give a brief overview with a few pointers to the relevant literature. \emph{List-coloring} (or \emph{choosability}) is a generalization of graph coloring introduced independently by Vizing \cite{Viz} and Erd\H{o}s, Rubin, and Taylor \cite{ERT}, which has by now become a classical part of graph coloring theory \cites[\S14.5]{BondyMurty}[\S5.4]{Diestel}. In list-coloring, the sets of available colors may vary from vertex to vertex, and the objective is to assign a color to each vertex from its list so that adjacent vertices receive different colors. \emph{DP-coloring} \ep{also known as \emph{correspondence coloring}} is a further generalization invented by Dvo\v{r}\'ak and Postle \cite{DP}. In the DP-coloring framework, not only the lists of available colors but also the identifications between them are allowed to vary from edge to edge. This notion is closely related to \emph{local conflict coloring} introduced by Fraigniaud, Heinrich, and Kosowski \cite{FHK} with a view toward applications in distributed computing. Even though DP-coloring has only emerged relatively recently, it has already attracted considerable attention.\footnote{According to MathSciNet, the original paper \cite{DP} by Dvo\v{r}\'ak and Postle has over 100 citations at the time of writing.} The \emph{paint number} of a graph generalizes list-coloring in a different way. It is an ``online'' variant of list-coloring wherein the lists of available colors are revealed in stages by an adversary, which was %. This notion was
    independently developed by Schauz \cite{Schauz} and Zhu \cite{Zhu}. %; see \todo{cite} for a sample of related results.
    Finally, the \emph{DP-paint number} is a common upper bound on the DP-chromatic number and the paint number, introduced and studied by Kim, Kostochka, Li, and Zhu \cite{Paint}.

    It turns out that, in contrast to $\de(G)$, the weak degeneracy of a graph enjoys various nontrivial upper bounds that yield corresponding results for the coloring parameters listed in Theorem~\ref{thm:wd+1}. For example, as mentioned above, all $d$-regular graphs have degeneracy exactly $d$. On the other hand, we have a version of Brooks' theorem for weak degeneracy: All connected $d$-regular graphs other than cycles and cliques are weakly $(d-1)$-degenerate \cite[Theorem 1.5]{WD}. (Note that both odd and even cycles have weak degeneracy $2$, which is a consequence of the fact that their DP-chromatic number is $3$ \cite{DP}.) Furthermore, for $d \geq 3$, a graph that is not weakly $(d-1)$-degenerate must contain either a $(d+1)$-clique or a somewhat ``dense'' subgraph: %---i.e., a of average degree at least $d + \epsilon$, where $\epsilon > 0$ depends only on $d$:

    \begin{thm}[{AB--EL \cite[Theorem 1.7]{WD}}]
        If $G$ is a nonempty graph with $\wde(G) \geq d \geq 3$, then either $G$ contains a $(d+1)$-clique, or it has a nonempty subgraph $H$ with average degree at least
        \[
            d + \frac{d - 2}{d^2 + 2d - 2} \,>\, d.
        \]
    \end{thm}

    In \cite{Regular}, Yang showed that for every $d$, there is a $d$-regular graph $G$ with $\wde(G) = \lfloor d/2 \rfloor + 1$. (This spectacularly disproved a pessimistic conjecture of the first two authors \cite[Conjecture 1.10]{WD}.) It is also known that $d$-regular graphs $G$ of girth at least $5$ satisfy $\wde(G) \leq d - \Omega(\sqrt{d})$ \cite[Theorem 1.12]{WD}, and we conjecture that the same asymptotic bound holds for triangle-free graphs.

    The results cited above show that weak degeneracy is more powerful than ordinary degeneracy when one is working with regular graphs. Planar graphs form another class of great interest in graph coloring theory \cites[\S10]{BondyMurty}[\S4]{Diestel}. While planar graphs are $4$-colorable by a famous theorem of Appel and Haken \cite{appel1989every,4CTNewProof,4CTFormal}, the optimal upper bound on the parameters $\chi_\ell(G)$, $\chi_{DP}(G)$, $\chi_P(G)$, and $\chi_{DPP}(G)$ for planar $G$ is $5$, established by Thomassen \cite{Thom}, Dvo\v{r}\'ak and Postle \cite{DP}, Schauz \cite{Schauz}, and Kim, Kostochka, Li, and Zhu \cite{Paint} respectively; the optimality was shown by Voigt \cite{Voigt}. This is another instance where the ``degeneracy plus one'' bound is not sharp: the best general upper bound on $\de(G)$ for planar $G$ is $5$, which results in the bound of $6$ for $\chi_\ell(G)$, $\chi_{DP}(G)$, etc. In contrast, the main result of this paper is a proof that ``weak degeneracy plus one'' does give the optimal bound:

    \begin{thm}\label{thm:main}
        Planar graphs are weakly $4$-degenerate.
    \end{thm}

    Theorem~\ref{thm:main} was stated by the first two authors in \cite[Theorem 1.4]{WD}. Unfortunately, as pointed out to us by Tao Wang \ep{personal communication}, the argument given in \cite{WD} is flawed, and we believe the flaw is fatal. Since the mistake is somewhat subtle (at least in our opinion) and there is a danger it would reappear in other papers on weak degeneracy, we feel it necessary to briefly explain what it is. The approach followed in \cite{WD} was to adapt Thomassen's famous proof that planar graphs are $5$-list-colorable \cite{Thom}. Thomassen's proof is inductive, and to facilitate the induction in one of the cases, it relies on removing a particular pair of colors from the lists of several vertices---an operation that has no analog in the weak degeneracy framework. The authors of \cite{WD} tried to get around this issue by using a certain monotonicity property of weak degeneracy, but in fact that property does not hold. To be more precise, here is the general statement they relied on:

    \begin{false}
        Let $G$ be a graph and let $f$, $f' \colon V(G) \to \N$ be functions such that $f(v) \leq f'(v)$ for all $v \in V(G)$. Suppose that starting with $(G,f)$, all vertices can be removed from $G$ by some sequence of legal applications of the $\delsave$ operation.

        %\smallskip
        
         {\upshape\textbf{True part:}} The same sequence of operations is legal when used starting with $(G,f')$.

        %\smallskip
        
         {\upshape\textbf{False part:}} Furthermore, if we run this sequence of operations starting with $(G,f')$, then at the time a vertex $w \in V(G)$ is deleted, the value of the function at $w$ is at least $f'(w) - f(w)$. %if $f'(v) > f(v)$ for some $v \in V(G)$, then at the step when $v$ is deleted, the value of the function at $v$ is strictly positive.
    \end{false}

    While the false part of the above ``Lemma'' may seem plausible at first glance, here is a simple counterexample. Suppose $G \cong K_2$ is a single edge $uw$, and let $f \defeq (u \mapsto 1, w \mapsto 0)$, $f' \defeq (u \mapsto 1, w \mapsto 1)$. Consider the sequence of operations $\delsave(u,w)$, $\del(w)$. The operation $\delsave(G,f,u,w)$ removes $u$ from $G$ and, since $f(u) > f(w)$, it does not alter the value of the function at $w$. On the other hand, we have $f'(u) = f'(w)$, so the operation $\delsave(G, f',u,w)$ brings the value at $w$ to $0$. Thus, when $w$ is deleted, the value at $w$ is $0$ in both cases, even though $f'(w) > f(w)$. This failure of monotonicity makes it difficult to implement in the weak degeneracy setting the idea of ``reserving'' colors for future use, which is common in graph coloring arguments.
    
    Nevertheless, we were able to find a different approach. It is still inductive and greatly inspired by Thomassen's argument. Even more precisely, the form of the inductive statement is adapted from \cite[Theorem 6]{DLM} by Dvo\v{r}\'ak, Lidick\'y, and Mohar. That being said, we should emphasize that our proof is not directly analogous to the argument in \cite{DLM}, which requires a coordinated choice of colors for certain vertices (see \cite[p.~59, (X4b)]{DLM}), a step that in general has no counterpart in the weak degeneracy setting. Nevertheless, after having established local structure of a purported minimal counterexample, a careful choice of deletion operations achieves the same effect. The bulk of our proof comprises a series of local reducible configurations that we show may not appear in such a counterexample.
    
    Although weak degeneracy has only been introduced relatively recently, it has already attracted considerable attention from researchers in graph coloring theory. In particular, it has been applied to studying planar graphs with restricted structure \cite{Planar1,Planar2,Planar3,Embedded2,Embedded3}, graphs embeddable in surfaces other than the plane \cite{Embedded1,Embedded2,Embedded3}, and other graph coloring variants \cite{Extra1,Extra2,Extra3}. We feel %the resulting argument
    that our proof of Theorem~\ref{thm:main} sheds new light on the nature of planar graph coloring and expect some of its ideas to find further applications, especially since the tools we employ are of necessity very flexible. %, since the tools we employ are of necessity very flexible.
    
    \section{The inductive statement}

     %, which we prove by induction:
    %\section{Proof of Theorem}\label{sec:proof}

    We begin with a few preliminary remarks. Given a graph $G$, a set of vertices $S$, and a vertex $v \in V(G)$, we let $N_S(v)$ be the set of all neighbors of $v$ in $S$ and write $\deg_S(v) \defeq |N_S(v)|$. When $G$ is a graph and $f \colon V(G) \to \N$ is a function, we say that $G$ is \emphd{\ep{strongly} $f$-degenerate} if all its vertices can be removed via a sequence of legal applications of the $\del$ operation.
    
    We shall often employ the following slight abuse of terminology. If $G$ is a graph, $f$ is an integer-valued function with $\dom(f) \supseteq V(G)$, and $f'$ is the restriction of $f$ to $V(G)$, we use the phrase ``$G$ is weakly $f$-degenerate'' to mean ``$G$ is weakly $f'$-degenerate,'' and write $\delsave(G,f,u,w)$ and $\del(G,f,u)$ for $\delsave(G, f', u, w)$ and $\del(G,f',u)$ respectively. 

    The following fact will be used repeatedly (this is the true part of the ``Lemma'' from the introduction):

    \begin{lemma}[{\cite[Lemma 2.1]{WD}}]\label{lemma:monotone}
        Let $G$ be a graph and let $f$, $f' \colon V(G) \to \N$ be functions such that $f(v) \leq f'(v)$ for all $v \in V(G)$. If $G$ is weakly $f$-degenerate, then $G$ is weakly $f'$-degenerate as well. 
    \end{lemma}

    We derive Theorem~\ref{thm:main} from a stronger technical statement \ep{as is common in arguments related to Thomassen's theorem, the stronger statement facilitates the induction}. As mentioned in the introduction, this particular statement is inspired by \cite[Theorem 6]{DLM}. %Recall that a \emphd{near-triangulation} is a plane graph in which every face except possibly the outer one is a triangle.
    For a plane graph $G$, its \emphd{outer face boundary} is the (not necessarily induced) subgraph $\bdry{G} \subseteq G$ whose vertices and edges are exactly the ones incident to the outer face of $G$. We say that vertices in a set $S \subseteq V(\bdry{G})$ are \emphd{consecutive} if $S = \0$ or the induced subgraph $(\bdry{G})[S]$ is connected.
    
\begin{thm}\label{thm:inductive}
    Let $G$ be a plane graph. %and let $C$ be the graph whose vertex- and edge-sets form %are precisely those of
    %the outer face boundary %walk
    %of $G$.
    Let $S \subseteq V(\bdry{G})$ be a set of at most three consecutive vertices %either a path with at most three vertices or a triangle,
    and let $I \subseteq V(\bdry{G}) \setminus S$ be a set that is independent in $G$. Let $f \colon V(G - S) \to \Z$ be the function given by %be a function whose domain includes $V(G)$ and is defined as follows: 
%\begin{itemize}
%        \item $f_0(v) = 0 $ if $v \in S$,
%        \item $f_0(v) = 4 $ if $v \in V(G) \setminus V(C)$,
%        \item $f_0(v) = 3$ if $v \in C \setminus (S \cup I)$
%        \item $f_0(v) = 2$ if $v \in I$
%\end{itemize}
    %\[
    %    f_0(v) \,\defeq\, \begin{cases}
    %        4 &\text{if } v \in V(G) \setminus V(C),\\
    %        3 &\text{if } v \in V(C) \setminus (S \cup I),\\
    %        2 &\text{if } v \in I,\\
    %        0 &\text{if } v \in S.
    %    \end{cases}
    %\]
     \begin{equation}\label{eq:f}
        f(v) \,\defeq\, \begin{cases}
            4 - \deg_S(v) &\text{if } v \in V(G) \setminus V(\bdry{G}),\\
            3 - \deg_S(v) &\text{if } v \in V(\bdry{G}) \setminus (S \cup I),\\
            2 - \deg_S(v) &\text{if } v \in I.%,\\
            %0 &\text{if } v \in S.
        \end{cases}
    \end{equation}
    %Let $f \colon V(G) \to \Z$ be the function %whose domain includes $V(G)$ and is
 %defined as $f(v) \defeq f_0(v)$ for all $v \in V(S)$, and $f(v) \defeq f_0(v) - |S \cap N(v)|$ for all $v \in V(G) \setminus V(S)$.
    Then the graph $G - S$ is weakly $f$-degenerate unless there exists a vertex $v \in I$ with $3$ neighbors in $S$. %$f(v) = -1$. 
\end{thm}

    Note that in the setting of Theorem~\ref{thm:inductive}, the condition that no vertex in $I$ has $3$ neighbors in $S$ is equivalent to saying that the function $f$ is nonnegative. %$f(v) \geq 0$ for all $v \in V(G - S)$. %, and for $v \in I$, we have $f(v) \geq 0$ unless $|S| = 3$ and $v$ is adjacent to every vertex in $S$, in which case $f(v) = -1$.

    \begin{proof}[Proof of Theorem~\ref{thm:main} from Theorem~\ref{thm:inductive}]
        Let $G$ be a plane graph. %and %Since adding vertices or edges cannot decrease the weak degeneracy of a graph, it suffices to prove the theorem when $G$ is a maximal planar graph on at least $3$ vertices, in which case $G$ is a plane triangulation.
        %let $C$ be its outer face boundary.
        Applying Theorem~\ref{thm:inductive} with $S = I = \0$, we see that $G$ is weakly $f$-degenerate, where $f(v) = 4$ for $v \in V(G) \setminus V(\bdry{G})$ and $f(v) = 3$ for $v \in V(\bdry{G})$. This implies that $G$ is weakly $4$-degenerate by Lemma~\ref{lemma:monotone}.
    \end{proof}

    \section{Proof of Theorem~\ref{thm:inductive}}

%======================================================================

%\textcolor{blue}{Maybe a little blurb about how we're going to delete a vertex or two near the outer cycle. Vertices INSIDE $C$ lose at most two from $f(v)$, and if they do lose 2, they join the new outer cycle and we add them to $I$. If a vertex in the interior of $C$ just loses one colour, it does not get added to $I$ but because of the structure of $G$ it gets added to $C'$; and if a vertex in $C$ loses one colour, it gets added to $I$. We then use induction on what's leftover, etc. Maybe explicitly say what the values of $f(v)-f'(v)$ are for things as we update them? Maybe not necessary, idk.}
%\begin{proof}

    \subsection{A counterexample and its basic properties}

Suppose Theorem~\ref{thm:inductive} fails and let $(G,S,I,f)$ be a counterexample. %$\sum_{v \in V(G)} f(v)$.
Explicitly, this means that:
\begin{itemize}
    \item $G$ is a plane graph, %and $C$ is its outer face boundary, % near-triangulation,
    %\item $C$ is the outer face boundary of $G$,
    \item $S \subseteq V(\bdry{G})$ is a set of at most $3$ consecutive vertices,
    \item $I \subseteq V(\bdry{G}) \setminus S$ is a set that is independent in $G$,
    \item $f \colon V(G - S) \to \Z$ is defined by \eqref{eq:f}, %\eve{(should the $\Z$ be $\N$ or vice-versa in 2.2 for consistency?)} \eugene{I prefer changing the prior $\N$ to $\Z$ since later it's mentioned that something is equivalent to $f$ being nonnegative, which seems odd if we restrict it to $\N$ in the first place.}
    \item %$f(v) \geq 0$ for all $v \in V(G - S)$, i.e.,
    no vertex in $I$ has $3$ neighbors in $S$,
    \item yet, the graph $G - S$ is not weakly $f$-degenerate.
\end{itemize}
%$f(v) \geq 0$ for all $v \in V(G - S)$, yet $G - S$ is not weakly $f$-degenerate.
 %, so we must have $|S| \geq 2$ (otherwise we could add a vertex to $S$ and reduce $\|f\|_1$). Moreover, $I$ must be a maximal independent set in $C - S$ (as otherwise we could add a vertex to $I$ and again reduce $\|f\|_1$).
We choose such a counterexample to minimize $|V(G)|$, then maximize $|E(G)|$, then maximize $|S|$, and then finally maximize $|I|$. %subject to that, %and, subject to that, $\|f\|_1 \defeq \sum_{v \in V(G - S)} f(v)$.
%The following claims are immediate consequences of the choice of the counterexample.
%Now we proceed via a series of claims.
The remainder of the argument comprises a series of claims describing the structure of our counterexample that finally culminates in a contradiction.

%\begin{claim}
%    $I$ is a maximal independent set in $C - S$.
%\end{claim}
%\begin{proof}
%    Otherwise we could increase $|I|$ while keeping $|V(G)|$ and $|S|$ unchanged. %and again reduce $\|f\|_1$.
%\end{proof}

\begin{claim}
    $G$ is connected.
\end{claim}
\begin{proof}
    Otherwise Theorem \ref{thm:inductive} would apply to each component of $G$, and if $H-S$ is weakly $f$-degenerate for each component $H$ of $G$, then so is $G$ itself.
\end{proof}

\begin{claim}\label{claim:S2}
    $|S| \geq 2$.
\end{claim}
\begin{proof}
    It is clear that $|V(G)| \geq 2$. Suppose that $|S| < 2$. If $S = \0$, then let $u \in V(\bdry{G})$ be an arbitrary vertex, and if $|S| = 1$, then let $u \in V(\bdry{G})$ be a neighbor of the vertex in $S$, which exists since $G$ is connected. Set $S' \defeq S \cup \set{u}$ and $I' \defeq I \setminus \set{u}$. Since $|S'| > |S|$, by the choice of our counterexample, Theorem~\ref{thm:inductive} holds with $G$, $S'$, and $I'$ in place of $G$, $S$, and $I$. Since $|S'| \leq 2$, no vertex in $I'$ can have $3$ neighbors in $S'$, so the graph $G - S' = G - S - u$ is weakly $f'$-degenerate, where for each $v \in V(G - S - u)$,
    \[
        f'(v) \,\defeq\, f(v) - \deg_{\set{u}} (v).
    \]
    As $\del(G - S,f,u) = (G - S - u, f')$, it follows that $G - S$ is weakly $f$-degenerate, a contradiction.
    %Therefore, if $|S| < 2$, we could add a vertex to $S$, thus increasing $|S|$ while keeping $|V(G)|$ unchanged. %and reduce $\|f\|_1$.
\end{proof}

    \subsection{Separating paths and $\ell$-chords}

Next, we introduce the notion of an {$\ell$-chord} in $G$, which will play a key role in the remainder of the proof. Informally, an $\ell$-chord is a path of length $\ell$ in $G$ that joins two vertices on the boundary of the outer face and breaks $G$ into two pieces. (Recall that the length of a path $P$ is $|E(P)| = |V(P)| - 1$.)

\begin{definition}[Separating paths and $\ell$-chords]\label{seppathdef}
    %Let $G$ be a connected plane graph with outer face boundary walk $C$, and let $S \subseteq C$.
    Let $P$ be a path in $G$ with endpoints $u$, $v \in V(\bdry{G})$. We say that \emphd{$P$ separates $G$ into graphs $G_1$ and $G_2$} if the following hold:
\begin{itemize}
    \item $G_1$ and $G_2$ are induced connected subgraphs of $G$ with the inherited plane embedding,
    \item $V(G_1) \cap V(G_2) = V(P)$, $V(G_1) \cup V(G_2) = V(G)$,  and $E(G_1) \cup E(G_2) = E(G)$,
    \item $|V(G_i)| < |V(G)|$ for each $i \in \set{1,2}$, and
    %\item for each $i \in \{1,2\}$, we have that $V(G) \setminus V(G_i) \neq \0$,
    \item $P$ is in the outer face boundary of both $G_1$ and $G_2$. %, and
    %\item $|V(G_1) \cap S| \geq |V(G_2) \cap S|$.
\end{itemize}
Throughout, we adopt the convention that the graphs $G_1$ and $G_2$ are chosen so that \begin{equation}\label{eq:choice}
    |V(G_1) \cap S| \,\geq\, |V(G_2) \cap S|.
\end{equation}
By Claim~\ref{claim:S2}, \eqref{eq:choice} in particular implies that $|V(G_1) \cap S| \geq 2$. If $P$ separates $G$, we call $P$ a \emphd{separating path}. A separating path of length $\ell$ is called an \emphd{$\ell$-chord}. %For $i \in \{0,1,2\}$, we call a separating of length $i$ an \emph{$i$-chord}.
See Figure~\ref{fig:2chord} for an illustration.
\end{definition}
\begin{figure}[t]
    \centering
       %ForestGreen, Fuchsia, BlueViolet, Sepia, Periwinkle, YellowGreen, WildStrawberry

\begin{tikzpicture}

        \node[invisible] (3) at (1.5,1.7) {};
        \node[invisible] (4) at (2.5,0.8) {};
        \node[invisible] (7) at (-1.5,-1.7) {};
        \node[invisible] (8) at (-2.5,-0.8) {};
        \node[white] (1) at (0,1.8) {$u$};
        \node[white] (2) at (1,1.8) {};
        \node[white] (9) at (0,0) {$x$};
        \node[white] (5) at (0,-1.8) {$v$};
        \node[white] (6) at (-1,-1.8) {};

        \node[] (10) at (2.1,1.5) {$\ddots$};
        \node[] (11) at (-2.1,-1.2) {$\ddots$};
        \node[] (12) at (1,0) {$G_2$};
        \node[] (13) at (-1,0) {$G_1$};

%\draw[black] (v3) to [out=180,in=80,looseness=1] (v4);
        \draw[black] (1)--(2); 
        \draw[black] (2)--(3);
        \draw[] (4) to [out=-80,in=0] (5); 
        \draw[black] (5)--(9); 
        \draw[black] (1)--(9); 
        \draw[black] (5)--(6); 
        \draw[black] (6)--(7); 
        \draw[] (8) to [out=100, in=180] (1); 

\end{tikzpicture}
       \caption{A 2-chord $uxv$. Note that the vertices $u$ and $v$ may be adjacent, in which case the edge $uv$ belongs to both $G_1$ and $G_2$. Both $G_1$ and $G_2$ must include at least one vertex other than $u$, $v$, and $x$.}\label{fig:2chord}
\end{figure}
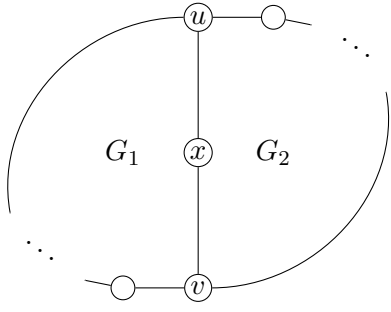

    In the sequel, we only employ Definition~\ref{seppathdef} with $\ell \in \set{0,1,2}$. Note that a $0$-chord is simply a cut vertex in $G$, and if $G$ is $2$-connected, then a $1$-chord is just a chord in the cycle $\bdry{G}$.

    A separating path $P$ splits the graph $G$ into two strictly smaller graphs $G_1$ and $G_2 $. Applying the inductive hypothesis to $G_1$ yields the following: %The next claim explains why separating paths are useful in analyzing weak degeneracy:

    \begin{claim}\label{claim:split}
        Let $P$ be a path that separates $G$ into $G_1$ and $G_2$. For each $v \in V(G_2 - S - V(P))$, define
        \[
            f'(v) \,\defeq\, f(v) - \deg_{V(P) \setminus S}(v).
        \]
        Then either $f'(v) < 0$ for some $v \in V(G_2 - S - V(P))$ or $G_2 - S - V(P)$ is not weakly $f'$-degenerate.
    \end{claim}
    \begin{proof}
        Since $|V(G_1)| < |V(G)|$ and $V(\bdry{G_1}) \supseteq V(\bdry{G}) \cap V(G_1)$, the choice of our counterexample and Lemma~\ref{lemma:monotone} show that the graph $G_1 - S$ is weakly $f$-degenerate, i.e., starting with $(G_1 - S, f)$, we can remove every vertex in $V(G_1) \setminus S$ via legal applications of the $\delsave$ operation. The same sequence of operations starting with $(G-S, f)$ yields the pair $(G_2 - S - V(P), f')$ (the operations remain legal because $G_1$ is an induced subgraph of $G$). If $f'(v) \geq 0$ for all $v$ and $G_2 - S - V(P)$ is weakly $f'$-degenerate, then we can remove all the remaining vertices by legal applications of $\delsave$, showing that $G - S$ is weakly $f$-degenerate, which is a contradiction. 
    \end{proof}

    Using Claim~\ref{claim:split} and then applying the inductive hypothesis to $G_2$, we can show that $G$ has no $0$-chords and no $1$-chords. This is done in the next two claims.

\begin{claim}\label{claim:2conn}
    $G$ is $2$-connected \ep{hence it does not have a $0$-chord}.
\end{claim}
\begin{proof}
    %Note that $G$ is connected, as otherwise Theorem \ref{thm:inductive} would apply to each of its components, and if $H-S$ is weakly $f$-degenerate for each component $H$ of $G$, then so is $G$ itself.
    We first note that $|V(G)| \geq 3$, since otherwise $V(G) = S$ by Claim~\ref{claim:S2} and the empty graph $G - S$ is trivially weakly $f$-degenerate. Now suppose $G$ has a cut vertex, i.e., a $0$-chord $u$ separating $G$ into $G_1$ and $G_2$. %, where $|V(G_1) \cap S| \geq |V(G_2) \cap S|$. %Since the vertices of $S$ follow consecutively on $C$, this implies that if $V(G_2) \cap S \neq \0$, then $u \in S$.
    As in Claim~\ref{claim:split}, for each $v \in V(G_2 - S - u)$, we let
    \[
        f'(v) \,\defeq\, f(v) - \deg_{\set{u} \setminus S} (v). %\quad \text{for all } v \in V(G_2 - S - u).
    \]
     %Since $G$ is a minimum counterexample,
    %Now we consider two cases.

    %\smallskip

    Suppose first that  $u \in S$. Then $f'(v) = f(v)$ for all $v \in V(G_2 - S)$. By the choice of our counterexample and Lemma~\ref{lemma:monotone}, $G_2 - S$ is weakly $f$-degenerate, in contradiction to Claim~\ref{claim:split}. %. But $G - S$ is a disjoint union of $G_1 - S$ and $G_2 - S$, so it is weakly $f$-degenerate as well.

    Now suppose $u \notin S$. Since the vertices of $S$ are consecutive and $|V(G_1) \cap S| \geq |V(G_2) \cap S|$ by \eqref{eq:choice}, we see that in this case $V(G_2) \cap S = \0$ and $f'(v) = f(v) - \deg_{\set{u}}(v) > 0$ for all $v \in V(G_2 - u)$. By the choice of our counterexample, Theorem~\ref{thm:inductive} holds with $G_2$, $\set{u}$, and $(I \cap V(G_2)) \setminus \set{u}$ in place of $G$, $S$, and $I$, i.e., $G_2 - u$ is weakly $f'$-degenerate. This again contradicts Claim~\ref{claim:split}.
\end{proof}

    Note that, by Claim \ref{claim:2conn}, the boundary of every face of $G$, in particular $\partial G$, is a cycle.

    \begin{claim}\label{claim:1chord} $G$ does not have a $1$-chord.
\end{claim}
\begin{proof}
    Suppose it does and let $uw$ be a 1-chord separating $G$ into $G_1$ and $G_2$. We choose $uw$ to maximize $|V(G_1)|$. Let $S' \defeq (S \cap V(G_2)) \cup \set{u,w}$ and $I' \defeq (I \cap V(G_2)) \setminus \set{u,w}$. As in Claim~\ref{claim:split}, for every vertex $v \in V(G_2 - S - u - w)$, we define
    \[
        f'(v) \,\defeq\, f(v) - \deg_{\set{u,w} \setminus S} (v) \,=\, \begin{cases}
            4 - \deg_{S'}(v) &\text{if } v \in V(G_2) \setminus V(\bdry{G_2}),\\
            3 - \deg_{S'}(v) &\text{if } v \in V(\bdry{G_2}) \setminus (S' \cup I'),\\
            2 - \deg_{S'}(v) &\text{if } v \in I'.
            \end{cases}%\quad \text{for all } v \in V(G_2 - S - u).
    \]
    %Set $S' \defeq (S \cap V(G_2)) \cup \set{u,w}$.
    %where $S' \defeq (S \cap V(G_2)) \cup \set{u,w}$ and $I' \defeq (I \cap V(G_2)) \setminus \set{u,w}$. \eve{(I think it'd be better to have the definition of $S'$ and $I'$ be before the sentence \emph{As in Claim 4, (...)}.)}
    Convention \eqref{eq:choice} implies that the vertices in $S'$ are consecutive on $\bdry{G_2}$ and $|S'| \leq 3$. Furthermore, if $|S'| = 3$, then no vertex in $V(\bdry{G_2})$ may be adjacent to all $3$ vertices in $S'$. Indeed, say $S' = \set{u,w,x}$ where, without loss of generality, $\set{w,x} \subseteq S$. If a vertex $y \in V(\bdry{G_2})$ is adjacent to $u$, $w$, and $x$, then $wy$ is a $1$-chord separating $G$ into graphs $G_1'$ and $G_2'$ with $V(G_1) \subset V(G_1')$, which contradicts the choice of $uw$ as a $1$-chord maximizing $|V(G_1)|$. Therefore, by the choice of our counterexample, we may apply Theorem~\ref{thm:inductive} with $G_2$, $S'$, and $I'$ %$(I \cap V(G_2)) \setminus \set{u,w}$
    in place of $G$, $S$, and $I$ to conclude that $G_2 - S' = G_2 - S - u - w$ is weakly $f'$-degenerate. This contradicts Claim~\ref{claim:split}. %, because if $S \cap V(G_2) \neq \0$, then $S \cap \set{u,w} \neq \0$ as well. 
    %
   % 
    %Note that $G_1-S$ and $f$ satisfy the hypotheses of Theorem \ref{thm:inductive}, and hence since $G$ is a minimum counterexample, $G_1-S$ is weakly $f$-degenerate. Let $S' := (S \cap G_2) \cup \{u,w\}$. Let $f':V(G_2)\setminus V(S') \rightarrow \mathbb{N}$ be defined as $f'(v) = f(v)-|N_G(v) \cap (\{u,w\} \setminus V(S))|$ for all $v \in V(G_2) \setminus V(S')$.  If $|V(S)| = 2$ or $|V(S)| = 3$ but neither $u$ nor $w$ is the interior vertex of $S$, it is immediately clear that no vertex in $G_2$ is adjacent to three vertices of $S'$. If $|V(S)| = 3$ and, without loss of generality, $u$ is the internal vertex of $S$, then since $uw$ was chosen to maximize $|V(G_1)|$ it follows similarly that no vertex in $G_2$ is adjacent to all three vertices in $S'$.  Thus $G_2-S'$ and $f'$ satisfy the hypotheses of Theorem \ref{thm:inductive} and thus, since $G$ is a minimum counterexample, we have that $G_2-S'$ is weakly $f'$-degenerate. But then $G-S$ is weakly $f$-degenerate, a contradiction.
\end{proof}

\subsection{Some corollaries of the absence of $1$-chords}

    The next few claims follow fairly easily from Claim~\ref{claim:1chord}, i.e., the fact that $G$ has no $1$-chord.

\begin{claim}\label{claim:I_maximal}
    $I$ is a maximal independent set in the graph $\bdry{G} - S$.
\end{claim}
\begin{proof}
    Suppose that $u$ is a vertex in $V(\bdry{G}) \setminus (S \cup I)$ such that the set $I \cup \set{u}$ is independent in $\bdry{G} - S$. Claim~\ref{claim:1chord} implies that then $I \cup \set{u}$ is also an independent set in $G$.
    By Claim~\ref{claim:1chord} again, $u$ may not be adjacent to $3$ vertices in $S$. Therefore, by the choice of our counterexample, we may apply Theorem~\ref{thm:inductive} with $G$, $S$, and $I \cup \set{u}$ in place of $G$, $S$, and $I$ and then invoke Lemma~\ref{lemma:monotone} to conclude that $G - S$ is weakly $f$-degenerate, a contradiction.  %Otherwise we could increase $|I|$ while keeping $|V(G)|$ and $|S|$ unchanged. %and again reduce $\|f\|_1$.
\end{proof}

    \begin{claim}\label{claim:inside_exists}
        $V(G) \neq V(\bdry{G})$.
    \end{claim}
    \begin{proof}
        Suppose $V(G) = V(\bdry{G})$. Since $G$ has no $1$-chord by Claim~\ref{claim:1chord}, it follows that $G = \bdry{G}$ is a cycle. As $S \neq \0$ by Claim~\ref{claim:S2}, $G - S$ is either the empty graph or a path. If $G - S$ has at most one vertex, it is $0$-degenerate. Otherwise, it is $1$-degenerate and $f'(v) \geq 1$ for all $v \in V(G - S)$. It follows that, in all cases, $G - S$ is (strongly) $f$-degenerate, a contradiction.
    \end{proof}

    \subsection{Short cycles in $G$}

    Our aim in this section is to describe the structure of $3$- and $4$-cycles in $G$. First, we note that they cannot contain vertices in their interior; in particular, every triangle in $G$ bounds a face.

\begin{claim}\label{claim:cycles}\label{claim:triangles}\label{claim:4cycles}
        $G$ has neither a triangle nor a $4$-cycle with a vertex in its interior.
\end{claim}
    \begin{proof}
        Toward a contradiction, suppose that $F$ is either a $3$- or a $4$-cycle in $G$ with at least one vertex in its interior. Pick an arbitrary vertex $a \in V(F)$ and set $S^* \defeq V(F) \setminus \set{a}$. Note that $|S^*| \leq 3$. Let $G'$ be obtained from $G$ by deleting the vertices in the interior of $F$ and let $G^*$ be the subgraph of $G$ induced by the vertices in the interior of $F$ together with $S^*$. Note that $|V(G')| < |V(G)|$ and, since $a \notin V(G^*)$, $|V(G^*)| < |V(G)|$ as well. %The assumption $S \neq V(T)$ implies that $G'' \neq G$.
 By the choice of our counterexample, $G'-S$ is weakly $f$-degenerate, so it is possible to delete all vertices of $G'-S$ via a sequence of legal $\delsave$ operations. Applying the same operations starting with $(G - S,f)$ yields the pair $(G^*-S^*, f^*)$, where %$f':V(G'')\setminus V(T) \rightarrow \mathbb{Z}$ is given by 
 \[f^*(v) \,\defeq\, f(v) - \deg_{V(F)\setminus S}(v) \,=\, 4 - \deg_{V(F)}(v) \quad \text{for all } v \in V(G^* - S^*).\]
    %Let $C^*$ be the outer boundary of $G^*$.
    Note that $S^*$ is a set of at most $3$ consecutive vertices in $V(\bdry{G^*})$. By the choice of our counterexample, we may apply Theorem~\ref{thm:inductive} with $G^*$, $S^*$, and $\0$ in place of $G$, $S$, and $I$ to conclude that $G^* - S^*$ is weakly $f'$-degenerate, where for all $v \in V(G^* - S^*)$,
        \[
        f'(v) \,\defeq\, \begin{cases}
            4 - \deg_{S^*}(v) &\text{if } v \in V(G^*) \setminus V(\bdry{G^*}),\\
            3 - \deg_{S^*}(v) &\text{if } v \in V(\bdry{G^*}) \setminus S^*.
        \end{cases}
     \]
     Since $N_G(a) \cap V(G^*) \subseteq V(\bdry{G^*})$, it follows that $f'(v) \leq f^*(v)$ for all $v \in V(G^* - S^*)$, and thus $G^* - S^*$ is weakly $f^*$-degenerate by Lemma~\ref{lemma:monotone}. Therefore, $G - S$ is weakly $f$-degenerate, a contradiction.
    \end{proof}

    Next we show that $G$ is a \emphd{near-triangulation}, i.e., a plane graph in which every face except possibly the outer one is a triangle. This follows from the fact that, subject to minimizing $|V(G)|$, we chose $G$ to maximize $|E(G)|$.

    \begin{claim}\label{claim:near-triangulation}
        $G$ is a near-triangulation.
    \end{claim}
    \begin{proof}
        Suppose not and let $F$ be a cycle of length $k \geq 4$ that bounds a non-outer face of $G$. If $V(F) \subseteq V(\bdry{G})$, then, since $G$ has no $1$-chord by Claim~\ref{claim:1chord}, $V(F) = V(\bdry(G))$ and hence $G = \bdry{G} = F$, contradicting Claim~\ref{claim:inside_exists}. %is a cycle. By Claim~\ref{claim:S2}, $S \neq \0$, so $G - S$ is either the empty graph or a path, which is easily seen to be (weakly) $f'$-degenerate. %If $|V(G - S)| \leq 1$, then $G - S$ is $0$-degenerate, hence it is weakly $f'$-degenerate. Otherwise, $f'(v) \geq 1$ for all $v \in V(G - S)$ and thus $G - S$ is weakly $f'$-degenerate because every path is $1$-degenerate. 
        Therefore, it must be that $V(F) \setminus V(\bdry{G}) \neq \0$, so we can
        %Therfore, we Since $G$ has no $1$-chord by Claim~\ref{claim:1chord}, it is impossible to have $V(F) \subseteq V(\bdry{G})$, so we may
        fix a cyclic ordering $v_1$, \ldots, $v_k$ of $V(F)$ with $v_1 \notin V(\bdry{G})$. We claim that there exists a pair of distinct vertices $u$, $w \in V(F)$ such that $uw \notin E(G)$ and %at least one of $u$, $w$ is not in $V(\bdry{G})$
        $\set{u,w} \not\subseteq V(\bdry{G})$. Indeed, if $v_1v_3 \notin E(G)$, then we can take $\set{u,w} \defeq \set{v_1, v_3}$. Otherwise, since the edge $v_1v_3$ must lie outside the face bounded by $F$, either $v_2$ is in the interior of the cycle $v_1 v_3 v_4 \ldots v_k$ or $v_4$ is in the interior of the cycle $v_1 v_2 v_3$, and, in either case, we can take $\set{u,w} \defeq \set{v_2, v_4}$. 
        %
        %
        %Next we define a pair of vertices $u$, $w$ as follows:
       % 
        %\begin{itemize}%[wide]
        %    \item If $v_1 v_3 \notin E(G)$, then we let $u \defeq v_1$ and $w \defeq v_3$.
        %    
        %    \item Otherwise, the edge $v_1v_3$ must lie outside the face bounded by $F$, and hence either $v_2$ lies in the interior of the cycle $v_1 v_3 v_4 \ldots v_k$ or $v_k$ lies in the interior of the cycle $v_1 v_2 v_3$. On both cases, $v_2 v_k \notin E(G)$ \eve{(I think if you want to argue $v_2$ lies in the interior of this cycle, you need to invoke the fact that $v_1$ is not in the outer face boundary of $G$, no?)} In this case $v_2 \notin V(\bdry{G})$ and $v_2 v_4 \notin E(G)$. Then we let $u \defeq v_2$ and $w \defeq v_4$.
        %\end{itemize}
        %
        %\noindent By construction, $u$ and $w$ are two vertices in $V(F)$ that are not adjacent in $G$ and $u \notin V(\bdry{G})$.
        Now let $G'$ be the plane graph obtained from $G$ by joining $u$ and $w$ by an edge inside the face bounded by $F$. Then $\bdry{G'} = \bdry{G}$ and, since at least one of $u$, $w$ is not in $V(\bdry{G})$, $I$ is an independent set in $G'$ and no vertex in $I$ is adjacent in $G'$ to $3$ vertices in $S$. As $|V(G')| = |V(G)|$ and $|E(G')| > |E(G)|$, our choice of the counterexample together with Lemma~\ref{lemma:monotone} show that $G' - S$ is weakly $f$-degenerate. Since $G - S$ is a subgraph of $G' - S$, it follows that $G - S$ is weakly $f$-degenerate as well, a contradiction.
    \end{proof}

    It follows immediately from Claims~\ref{claim:cycles} and \ref{claim:near-triangulation} that every $4$-cycle in $G$ has a chord. We can now argue that $\bdry{G}$ is a cycle of length at least $5$; moreover, we can strengthen Claim~\ref{claim:S2} and show that $|S| = 3$:

    \begin{claim}\label{claim:5}
        $\bdry{G}$ is a cycle of length at least $5$ and $|S| = 3$.
    \end{claim}
    \begin{proof}
        Since $V(G) \neq V(\bdry{G})$ by Claim~\ref{claim:inside_exists}, there exists a vertex in the interior of the cycle $\bdry{G}$. It follows by Claim~\ref{claim:cycles} that the length of $\bdry{G}$ is at least $5$. Next we argue that $|S| = 3$. We know that $|S| \geq 2$ by Claim~\ref{claim:S2}. Suppose $|S| = 2$ and let $u \in V(\bdry{G}) \setminus S$ be a neighbor of a vertex in $S$. Set $S' \defeq S \cup \set{u}$ and $I' \defeq I \setminus \set{u}$. Since $|S'| > |S|$, by the choice of our counterexample, Theorem~\ref{thm:inductive} holds with $G$, $S'$, and $I'$ in place of $G$, $S$, and $I$. By Claim~\ref{claim:1chord}, no vertex in $I'$ can be adjacent to all $3$ vertices in $S'$, so the graph $G - S' = G - S - u$ is weakly $f'$-degenerate, where for each $v \in V(G - S - u)$,
    \[
        f'(v) \,\defeq\, f(v) - \deg_{\set{u}} (v).
    \]
    As $\del(G - S,f,u) = (G - S - u, f')$, it follows that $G - S$ is weakly $f$-degenerate, a contradiction. %To complete the proof of the claim, we observe that since $S$ is a set of exactly $3$ consecutive vertices on the cycle $\bdry{G}$, $(\bdry{G})[S]$ is a path of length $2$, and $G[S] = (\bdry{G})[S]$ by Claim~\ref{claim:1chord}.
    \end{proof}

    From this point on, we use Claim~\ref{claim:5} to list the vertices of the cycle $\bdry{G}$ in their cyclic order as \[u_1,\ u_2,\ u_3,\ v_1,\ v_2,\ \ldots,\ v_t,\] where $S = \set{u_1, u_2, u_3}$. Here $t \defeq |V(\bdry{G})| - 3 \geq 2$ by Claim~\ref{claim:5}. We also let $v_{t+1} \defeq u_1$. %we let $S = \set{u_1, u_2, u_3}$, where $u_1 u_2$, $u_2 u_3 \in E(G)$.

    \subsection{$2$-chords in $G$ are special}

    Now we turn our attention to the structure of the $2$-chords in $G$. Although we cannot simply show they do not exist, we argue that they must have a very special form (see Figure \ref{fig:2chordsspecial}):

\begin{figure}[t]
        \centering
       %ForestGreen, Fuchsia, BlueViolet, Sepia, Periwinkle, YellowGreen, WildStrawberry

\begin{tikzpicture}
        \node[starnode] (1) at (-1.5,1.6) {$u_1$};
        \node[starnode] (2) at (0,1.8) {$u_2$};
        \node[starnode] (3) at (1.5,1.6) {$u_3$};
        \node[white] (4) at (0,0) {$y$};
        \node[white] (5) at (0,-1.8) {$z$};
        \node[invisible] (6) at (2.2,1.3) {};
        \node[invisible] (7) at (3,0.7) {};
        \node[invisible] (8) at (-2.2,1.3) {};
        \node[invisible] (9) at (-3,0.7) {};

        \node[] (10) at (2.6,1.1) {$\ddots$};
        \node[] (11) at (-2.5,1.1) {$.$};
        \node[] (14) at (-2.65,1) {$.$};
        \node[] (15) at (-2.8,0.9) {$.$};
        \node[] (12) at (1,0) {$G_2$};
        \node[] (13) at (-1,0) {$G_1$};

%\draw[black] (v3) to [out=180,in=80,looseness=1] (v4);
        \draw[black] (1)--(2); 
        \draw[black] (2)--(3);
        \draw[black] (2)--(4);
        \draw[black] (4)--(5);
        \draw[black] (2)--(3);
        \draw[] (9) to [out=220,in=180] (5); 
        \draw[] (7) to [out=-40, in=0] (5); 
        \draw[black] (3)--(6); 
        \draw[black] (8)--(1); 

\end{tikzpicture}
\begin{tikzpicture}
        \node[white] (1) at (-1.5,1.6) {$x$};
        \node[black] (2) at (0,1.8) {\textcolor{white}{$a$}};
        \node[white] (3) at (1.5,1.6) {$z$};
        \node[white] (4) at (0,0) {$y$};
        \node[invisible] (5) at (0,-1.8) {};
        \node[invisible] (6) at (2.2,1.3) {};
        \node[invisible] (7) at (3,0.7) {};
        \node[invisible] (8) at (-2.2,1.3) {};
        \node[invisible] (9) at (-3,0.7) {};

        \node[] (10) at (2.6,1.1) {$\ddots$};
        \node[] (11) at (-2.5,1.1) {$.$};
        \node[] (14) at (-2.65,1) {$.$};
        \node[] (15) at (-2.8,0.9) {$.$};
        %\node[] (12) at (1,0) {$G_2$};
        \node[] (13) at (0,-1.2) {$G_1$};

%\draw[black] (v3) to [out=180,in=80,looseness=1] (v4);
        \draw[black] (1)--(2); 
        \draw[black] (2)--(3);
        \draw[black] (2)--(4);
        \draw[black] (4)--(1);
        \draw[black] (4)--(3);
        \draw[black] (2)--(3);
        \draw[] (9) to [out=220,in=180] (5); 
        \draw[] (7) to [out=-40, in=0] (5); 
        \draw[black] (3)--(6); 
        \draw[black] (8)--(1);

\end{tikzpicture}
       \caption{The two cases described in Claim \ref{claim:2chord}. On the left, vertices in $S$ are shown as $4$-pointed stars and $x = u_2$. On the right, $a \in I$.}\label{fig:2chordsspecial}
\end{figure}
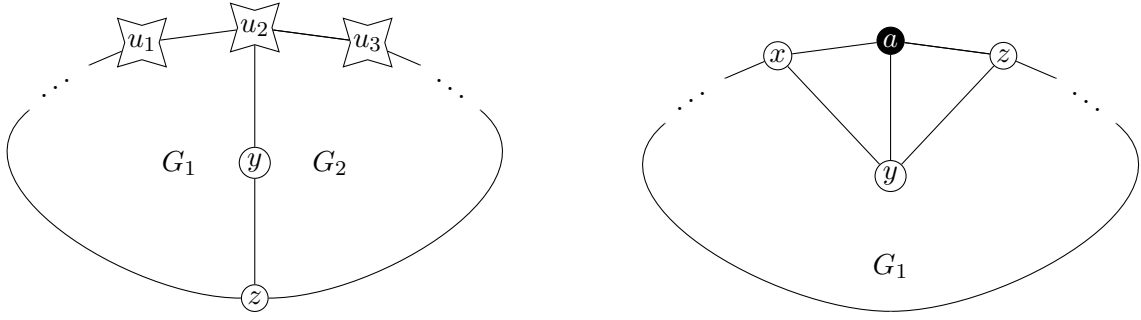

    \begin{claim}\label{claim:2chord}
    Let $xyz$ be a $2$-chord separating $G$ into $G_1$ and $G_2$. Then either $u_2 \in \set{x,z}$ or %$S$ is a path with three vertices and one of $x$ and $z$ is the internal vertex of $S$, or else
    %$xyz$ separates $G$ into two graphs $G_1$ and $G_2$, where
        %$I \cap V(G_2)$ contains a vertex %$i \in I$
        %adjacent to $x$, $y$, and $z$.
        there exists a vertex $a \in I$ such that $V(G_2) = \set{x,y,z,a}$ and $E(G_2) = \set{xy, yz, ax, ay, az}$. 
    \end{claim}
    \begin{proof}
        Suppose $u_2 \notin \set{x,z}$. By our convention \eqref{eq:choice}, this implies that $S \subseteq V(G_1)$ and $S \cap V(G_2) \subseteq \set{x,z}$. For every vertex $v \in V(G_2 - x - y - z)$, define
        \[
            f'(v) \,\defeq\, f(v) - \deg_{\set{x,y,z} \setminus S} (v). %\quad \text{for all } v \in V(G_2 - S - u).
        \] %not. The 2-chord $xyz$ separates $G$ into two graphs $G_1$ and $G_2$; and since neither $x$ nor $z$ is an internal vertex of $S$, we have that $S \subseteq V(G_1)$.
        If no vertex in $I \cap V(G_2)$ is adjacent to $x$, $y$, and $z$, then, by the choice of our counterexample, we may apply Theorem~\ref{thm:inductive} with $G_2$, $\set{x,y,z}$, and $(I \cap V(G_2)) \setminus \set{x,z}$ in place of $G$, $S$, and $I$ to conclude that $G_2 - x - y - z$ is weakly $f'$-degenerate, in contradiction to Claim~\ref{claim:split}. Therefore, there is a vertex $a \in I \cap V(G_2)$ adjacent to $x$, $y$, and $z$. Since $G$ has no $1$-chord by Claim~\ref{claim:1chord}, $ax$, $az \in E(\bdry{G})$. Furthermore, by Claim~\ref{claim:cycles}, the triangles $axy$ and $azy$ contain no vertices in their interiors. It follows that $V(G_2) = \set{x,y,z,a}$ and $E(G_2) = \set{xy, yz, ax, ay, az}$, as claimed. %Since $G$ is a minimum counterexample, $G_1-S$ is weakly $f$-degenerate. Let $f':V(G_2)\setminus \{x,y,z\} \rightarrow \mathbb{N}$ be defined as $f'(v) := f(v)-|N_G(v) \cap (\{x,y,z\} \setminus V(S))|$. Note that by assumption, no vertex in $G_2 \setminus \{x,y,z\}$ is adjacent to all three of $x,y,$ and $z$, and thus no vertex $v$ of $G'$ has $f'(v) < 0$. Since $G$ is a minimum counterexample, it follows that $G'-xyz$ is weakly $f'$-degenerate (here, $xyz$ is playing the role of $S$ in the theorem statement). But then $G$ is weakly $f$-degenerate, a contradiction.
    \end{proof}

\subsection{The structure around $v_2$ and $v_3$}

\begin{figure}[t]
        \centering
       %ForestGreen, Fuchsia, BlueViolet, Sepia, Periwinkle, YellowGreen, WildStrawberry

\begin{tikzpicture}

        \node[invisible] (9) at (-5,4.6) {};
        \node[starnode] (11) at (-4.2,4.6) {$u_1$};
        \node[starnode] (10) at (-3.1,4.5) {$u_2$};
        \node[starnode] (1) at (-2,4.3) {$u_3$};
        \node[white] (2) at (-1,4) {$v_1$};
        \node[black] (3) at (0.1,3.6) {\textcolor{white}{$v_2$}};
        \node[white] (4) at (1.1,3.1) {$v_3$};
        \node[black] (5) at (2.1,2.5) {\textcolor{white}{$v_4$}};
        \node[white] (8) at (3,1.7) {$v_5$};
        \node[invisible] (12) at (3.4,1.2) {};
        \node[white] (6) at (-0.7,2.2) {$x$};

        \draw[black] (1)--(2); 
        \draw[black] (9)--(11); 
        \draw[black] (2)--(3); 
        \draw[black] (3)--(4); 
        \draw[black] (4)--(5); 
        \draw[black] (5)--(8); 
        \draw[black] (2)--(6); 
        \draw[black] (6)--(3); 
        \draw[black] (6)--(4); 
        \draw[black] (1)--(10); 
        \draw[black] (10)--(11); 
        \draw[black] (8)--(12); 

\end{tikzpicture}
       \caption{The desired structure of the neighborhood of $v_2$. Here $x$ is not in $V(\bdry{G})$. The only neighbors of $x$ in $V(\bdry{G})$ are $v_1$, $v_2$, and $v_3$. Vertices in $S$ are shown as $4$-pointed stars. Vertices in $I$ are shown in black. %Since $G$ is a near-triangulation with no 1-chords, it follows that the vertex $y$ exists and $\{v_3,v_4,x\} \subseteq N(y)$.
       }\label{fig:x}
%\eugene{The positioning of this figure is pretty confusing, since under the "active assumptions" at this point in the prose the situation around $I$ is actually different.}
\end{figure}
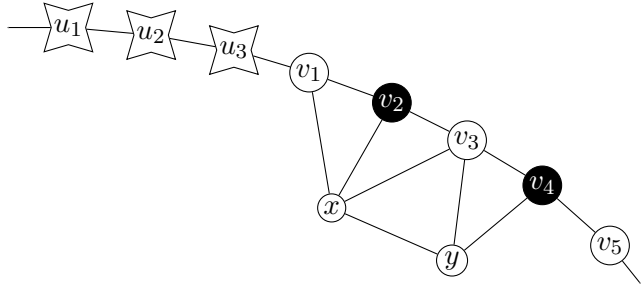

    At this point our aim becomes to precisely determine the structure of the graph $G$ and the set $I$ in the neighborhood of the vertices $v_2$ and $v_3$. Specifically, we will prove that the picture around these vertices is as shown in Figure~\ref{fig:x}. Crucially, $\{v_2,v_4\} \subseteq I$, and $v_2$ has degree $3$ in $G$, which will allow us to handle this vertex in the final stage of the proof.

    \begin{claim}\label{claim:in}
        %$\partial G$ is a cycle of length at least $6$ \ep{i.e., $t \geq 3$} and
        $v_1 \notin I$, $v_2 \in I$, $v_3 \notin I$.
    \end{claim}
    \begin{proof}
        Recall that $I$ is a maximal independent set in the graph $\bdry{G} - S$ by Claim~\ref{claim:I_maximal}. It follows that at least one of $v_1$, $v_2$ is in $I$, for otherwise $I \cup \set{v_1}$ would be a larger independent set. Thus, to establish the claim, we only need to argue that $v_1 \notin I$, which would imply $v_2 \in I$ and, since $I$ is independent, $v_3 \notin I$.
    %We first argue that $v_1 \notin I$.     
    
    Toward a contradiction, suppose that $v_1 \in I$. Then $v_2 \notin I$ because $I$ is independent. By Claim~\ref{claim:1chord}, $v_2$ is not adjacent to $u_2$ and $u_3$, so $f(v_2) \geq 2$ (and if $t \geq 3$, then $f(v_2) = 3$). Similarly, $v_1$ has exactly one neighbor in $S$, namely $u_3$, so $f(v_1) = 1$.

    \smallskip
    
    \underline{\emph{Case} $1$:} $v_3 \notin I$ (this includes the possibility that $v_3 = u_1 \in S$).

    \smallskip

    In this case, let $(G - S - v_1,f')\defeq \del(G-S, f, v_1)$ and $I'\defeq (I \setminus \set{v_1}) \cup \set{v_2}$. Since $G$ has no 1-chord by Claim \ref{claim:1chord} and $v_3 \notin I$, it follows that $I'$ is an independent set and no vertex in $I'$ has $3$ neighbors in $S$. Moreover, by Claim~\ref{claim:1chord} again, if $v \in V(G - S - v_1)$ is a vertex such that $f'(v) = f(v) - 1$, then either $v = v_2$ or $v \in V(\bdry{(G-v_1)}) \setminus V(\bdry{G})$. Therefore, Theorem~\ref{thm:inductive} with $G - v_1$, $S$, and $I'$ in place of $G$, $S$, and $I$ shows that $G - S - v_1$ is weakly $f'$-degenerate. Then $G-S$ is weakly $f$-degenerate, a contradiction.

    \smallskip
    
    \underline{\emph{Case} $2$:} $v_3 \in I$.

    \smallskip
    
    Let $(G -S - v_2, f') \defeq \delsave(G-S,f,v_2,v_3)$ and $(G - S - v_1 - v_2, f'') \defeq \del(G - S - v_2, v_1)$. %, and let $G^* \defeq G - v_1 - v_2$. %, so that $G'' = G^* - S$.
    Since $v_3 \in I$, we have $f(v_3) \leq 2 < 3 = f(v_2)$, and so $f'(v_3) = f(v_3)$. Moreover, by Claim~\ref{claim:1chord}, $v_3v_1 \notin E(G)$, so $f''(v_3) = f(v_3)$ as well. Since, by Claim~\ref{claim:1chord} again, $v_3$ is the only vertex in $V(\bdry{G}) \setminus (S \cup \set{v_1, v_2})$ that belongs to the set $N_G(v_1) \cup N_G(v_2)$, we conclude that for each $v \in V(G - S - v_1 - v_2)$,
    \[
        f''(v) \,=\, \begin{cases}
                    f(v) - \deg_{\set{v_1, v_2}}(v) &\text{if } v \notin V(\bdry{G}),\\
                    f(v) &\text{if } v \in V(\bdry{G}).
                    \end{cases}
    \]

    \smallskip
    
    \underline{\emph{Subcase} $2.1$:} $G$ has no vertex adjacent to both $v_1$ and $v_2$.

    \smallskip
    
     %If $G$ has no vertex adjacent to both $v_2$ and $v_1$, then, for all $v \in V(G'')$,
    %\[         f''(v) \,=\, \begin{cases}
     %           f(v) - 1 &\text{if $v$ is adjacent to $v_1$ or $v_2$ and $v \notin V(\bdry{G})$},\\
      %          f(v) &\text{otherwise}.
      %      \end{cases}
      %  \]
        In this case, if $v \in V(G - S - v_1 - v_2)$ is a vertex such that $f''(v) < f(v)$, then $f''(v) = f(v) - 1$ and $v \in V(\bdry{(G-v_1 - v_2)}) \setminus V(\bdry{G})$. Therefore, we may apply Theorem~\ref{thm:inductive} with $G - v_1 - v_2$, $S$, and $I \setminus \set{v_1}$ in place of $G$, $S$, and $I$ to conclude that $G-S - v_1 - v_2$ is weakly $f''$-degenerate. But then $G - S$ is weakly $f$-degenerate, a contradiction. %$G''-S, I\setminus \{v_1\},$ and $f''$ satisfy the hypotheses of Theorem \ref{thm:inductive} and so $G''-S$ is weakly $f''$-degenerate. But then $G-S$ is weakly $f$-degenerate, a contradiction.

        \smallskip
    
    \underline{\emph{Subcase} $2.2$:} $G$ has a vertex $u$ adjacent to both $v_1$ and $v_2$.

    \smallskip
        
        By Claim~\ref{claim:triangles} such a vertex $u$ is unique. Let $I' \defeq (I\setminus \{v_1\}) \cup \{u\}$. If $I'$ is an independent set and no vertex in $I'$ is adjacent to all $3$ vertices in $S$, then we may apply Theorem~\ref{thm:inductive} with $G - v_1 - v_2$, $S$, and $I'$ in place of $G$, $S$, and $I$ to again conclude that $G - S - v_1 - v_2$ is weakly $f''$-degenerate. Thus, either $u$ is adjacent to a vertex $u' \in I\setminus \set{v_1}$, or $u$ is adjacent to $u_1$, $u_2$, and $u_3$. In the former case, $v_1uu'$ is a 2-chord that violates Claim \ref{claim:2chord}, since $u_2 \notin \set{v_1, u'}$ and $v_2 \notin I$. % note $v_2 \not \in I$ by assumption, since $v_1 \in I$ and $I$ is independent, and that by Claim \ref{claim:1chord}, $G$ has no 1-chords
        In the latter case, $v_1uu_1$ is similarly a 2-chord that violates Claim \ref{claim:2chord}.
    \end{proof}

%\begin{case}\label{case:ni}
%The vertex $v_2$ is in $I$.
%\end{case}

%\begin{figure}[!h]
%    \centering
%       \input{figs/firstcase}
%       \caption{Black vertices are in $I$. Either $C$ has at most three vertices $v_1,v_2,v_3$ outside $S$, or $C$ has at least four vertices $v_1, \dots, v_4$ outside $S$ but $v_4 \not \in V(S) \cap I$.}\label{fig:firstcase}
%\end{figure}

    \begin{claim}\label{claim:8}
        $\partial G$ is a cycle of length at least $8$ \ep{i.e., $t \geq 5$} and $v_4 \in I$, $v_5 \notin I$.
    \end{claim}
    \begin{proof}
        We first note that $t \geq 3$. Indeed, if $t = 2$, i.e., $V(\bdry{G}) = \set{u_1, u_2, u_3, v_1, v_2}$, then $v_1 \notin I$ and $v_2 \in I$ by Claim~\ref{claim:in}. But we may reverse the ordering of the vertices on $\bdry{G}$, switching the roles of $v_1$ and $v_2$. As a result, we will have that $v_1 \in I$, contradicting Claim~\ref{claim:in}.
        
        Next, suppose $v_4 \notin I$. %Note that since $I$ is independent, it follows that $v_1 \not \in I$. If $C$ has at most three vertices $v_1,v_2,v_3$ outside $S$ or if $C$ has at least four vertices $v_1, \dots, v_4$ outside $S$ but $v_4 \not \in V(S) \cap I$, then
        Let $(G - S - v_2, f') \defeq \del(G-S, f, v_2)$ %See Figure \ref{fig:firstcase} for an illustration of these cases.
        and $I' \defeq (I \setminus \{v_2\}) \cup (\{v_1, v_3\})$. %, and let $C'$ be the graph whose vertex- and edge-set are precisely those of the outer face boundary walk of $G'$.
        Since $v_4 \notin I$ and $G$ has no $1$-chord by Claim~\ref{claim:1chord}, the set $I'$ is independent. The absence of $1$-chords in $G$ also shows that no vertex in $I'$ is adjacent to all $3$ vertices in $S$. Since the only neighbors of $v_2$ in $V(\bdry{G})$ are $v_1$ and $v_3$, if $v \in V(G - S - v_2)$ satisfies $f'(v) = f(v) - 1$, then either $v \in \set{v_1, v_3}$ or $v \in V(\bdry{(G - v_2)}) \setminus V(\bdry{G})$. Hence, we may apply Theorem~\ref{thm:inductive} with $G - v_2$, $S$, and $I'$ in place of $G$, $S$, and $I$ to conclude that $G - S - v_2$ is weakly $f'$-degenerate. But then $G - S$ is weakly $f$-degenerate, which is a contradiction.

        Therefore, $v_4 \in I$ (and, as a consequence, $t \geq 4$). If $t = 4$, then we may again reverse the ordering of the vertices on $\bdry{G}$, switching the roles of $v_1$ and $v_4$. As a result, we will have $v_1 \in I$, contradicting Claim~\ref{claim:in}. Hence, $t \geq 5$. Finally, we have $v_5 \notin I$ since $v_4 \in I$ and $I$ is independent.
        %
       % 
        %Note that $V(C') \setminus V(C) \subseteq N_G(v_2)$. Moreover, no vertex $x$ in $C'$ is adjacent to three vertices in $S$, as otherwise either $xu_2$ is a 1-chord of $C$ (contradicting Claim \ref{claim:1chord}), or $u_1xu_3$ is a 2-chord which, together with the existence of $v_1$ and the fact that $G$ has no 1-chord (Claim \ref{claim:1chord}), contradicts Claim \ref{claim:2chord}. No vertex in $I'$ is adjacent to two vertices in $S$ by similar reasoning. Finally, we claim $I'$ is an independent set: this follows from the definition of $G'$. Thus $G'\cup S$ and $I'$ satisfy the hypotheses of Theorem \ref{thm:inductive}, and so $G'$ is weakly $f'$-degenerate. But then $G-S$ is weakly $f$-degenerate, a contradiction.
    \end{proof}

    As a consequence of Claims~\ref{claim:in}, \ref{claim:8}, and \ref{claim:1chord}, we have
    \[
        f(v_1) \,=\, 2, \quad f(v_2) \,=\, 2, \quad f(v_3) \,=\, 3, \quad f(v_4) \,=\, 2, \quad f(v_5) \,\geq\, 2.
    \]
    In the next claim we locate the vertex $x$ from Figure~\ref{fig:x}.

    \begin{claim}\label{claim:x}
        There exists a unique vertex $x \in V(G) \setminus V(\bdry{G})$ that is adjacent to $v_1$, $v_2$, and $v_3$.
    \end{claim}
    \begin{proof}
        Let $(G - S - v_3, f') \defeq \delsave(G-S, f, v_3,v_4)$ and $(G - S - v_2 - v_3,f'') \defeq \del(G -S - v_3, f', v_2)$. %Note that there exists at most one vertex $x \in V(G) \setminus V(C)$ adjacent to both $v_3$ and $v_2$ by Claim \ref{claim:triangles}. 
        Since $f(v_3) = 3 > 2 = f(v_4)$ and $v_2 v_4 \notin E(G)$ by Claim~\ref{claim:1chord}, we have $f''(v_4) = f(v_4)$. Also, $f''(v_1) = 1$ and $f''(v) = f(v)$ for all $v \in V(\bdry{G}) \setminus (S \cup \set{v_1, v_2, v_3, v_4})$. To summarize, for all $v \in V(G - S - v_2 - v_3)$,
        \[
        f''(v) \,=\, \begin{cases}
                    f(v) - \deg_{\set{v_2, v_3}}(v) &\text{if } v \notin V(\bdry{G}),\\
                    f(v) - 1 &\text{if } v = v_1,\\
                    f(v) &\text{if } v \in V(\bdry{G}) \setminus \set{v_1}.
                    \end{cases}
        \]
        
        If there is no vertex in $V(G) \setminus V(\bdry{G})$ adjacent to both $v_2$ and $v_3$, then we apply Theorem~\ref{thm:inductive} with $G - v_2 - v_3$, $S$, and $(I \setminus \set{v_2}) \cup \set{v_1}$ in place of $G$, $S$, and $I$ to conclude that $G - S - v_2 - v_3$ is weakly $f''$-degenerate, which implies that $G - S$ is weakly $f$-degenerate, a contradiction. %Let $I'\defeq I \cup \{v_1\} \setminus \{v_2\}$. Note that no vertex in $I'$ is adjacent to two vertices in $S$ since $C$ has no 1-chord by Claim \ref{claim:1chord}. Moreover, $I$ is independent for the same reason. Thus $G''-S, I',$ and $f''$ satisfy the hypotheses of Theorem \ref{thm:inductive}. Since $(G,f)$ is a minimum counterexample, $G''$ is weakly $f''$-degenerate. But then $G-S$ is weakly $f$-degenerate, a contradiction. 
        Therefore, there is a vertex $x \in V(G) \setminus V(\bdry{G})$ adjacent to both $v_2$ and $v_3$, and it is unique by Claim~\ref{claim:cycles}. Suppose $x$ is not adjacent to $v_1$. Let $I' \defeq (I \setminus \set{v_2}) \cup \{v_1, x\}$. Note that $x$ is not adjacent to all $3$ vertices in $S$, as otherwise $u_1xu_3$ would be a 2-chord violating Claim~\ref{claim:2chord}. Moreover, $I'$ is an independent set. Indeed, if $I'$ is not independent, then, since $xv_1 \notin E(G)$ and $G$ has no $1$-chords, $x$ must be adjacent to a vertex $u \in I \setminus \set{v_2}$. But then $v_2xu$ is a 2-chord that violates Claim \ref{claim:2chord}. Therefore, applying Theorem~\ref{thm:inductive} with $G - v_2 - v_3$, $S$, and $I'$ in place of $G$, $S$, and $I$ again yields that $G - S - v_2 - v_3$ is weakly $f''$-degenerate, a contradiction.
    \end{proof}

    From now on, we let $x$ be the vertex given by Claim~\ref{claim:x}. Note that, since the triangles $v_1 v_2 x$ and $v_2 v_3 x$ contain no vertices in their interiors by Claim~\ref{claim:cycles}, we have $N_G(v_2) = \set{v_1, x, v_3}$.

    \begin{claim}\label{claim:no_other_neighbors}
        $x$ is not adjacent to any vertex in $V(\bdry{G})$ apart from $v_1$, $v_2$, and $v_3$.
    \end{claim}
    \begin{proof}
        If $x$ has a neighbor $u \in V(\bdry{G})\setminus (S \cup \set{v_1, v_2, v_3})$, then $v_2xu$ is a 2-chord that violates Claim \ref{claim:2chord}. %(noting $v_3 \not \in I$).
        Similarly, if $x$ is adjacent to $u_i$ with $i \in \set{1, 3}$, then $v_2 x u_i$ is a 2-chord that violates Claim~\ref{claim:2chord} (here we use that neither $v_1$ nor $v_3$ is in $I$). %, together with $v_1$ (if $u = u_3$) or $v_3$ (if $u = u_1$) violates Claim \ref{claim:2chord}.
        Finally, suppose that $x$ is adjacent to $u_2$. By Claim~\ref{claim:cycles}, the $4$-cycle $u_2 u_3 v_1 x$ has no vertex in its interior. Since $x$ is not adjacent to $u_3$ and $v_1$ is not adjacent to $u_2$ (because $G$ has no $1$-chords), this $4$-cycle bounds a face of $G$, which is impossible by Claim~\ref{claim:near-triangulation}. %, it follows that either $v_1 u_2$ or $xu_3$ is an edge of $G$. The former  If $S$ has three vertices and $x$ is adjacent to $u_2$, then by Claim \ref{claim:4cycles}, $u_2u_3v_1xu_2$ has no vertices in its interior.  Since $G$ is a near-triangulation and $x$ is not adjacent to $u_3$, we get that $u_2v_1$ in an edge, which contradicts Claim \ref{claim:1chord}.  Thus since $x \in V(G) \setminus V(C)$ and $N_G(x) \cap V(S) = \emptyset$, we have that $f(x) = 4$. 
    \end{proof}

   Since triangles in $G$ have empty interiors by Claim~\ref{claim:cycles}, $N_G(v_2) = \set{v_1, x, v_3}$. % and $N_G(v_3) = \set{v_2, x, y, v_4}$. 
   We have now achieved the structure shown in Figure~\ref{fig:x} and are ready for the denouement of our proof.

   Let $(G-S-v_3,f') = \delsave(G-S,f,v_3,v_4)$ and $(G-S-v_3-v_1,f''):= \delsave(G-S-v_3,f',v_1,v_2)$. Since $f(v_2) = 3$ by Claim \ref{claim:in}, we have that $f'(v_2) = 1$ and hence $f'(v_2) < 2 = f(v_1) = f'(v_1)$ as $\partial G$ has no 1-chord by Claim \ref{claim:1chord}.  Again by Claim \ref{claim:1chord}, it follows that for all $v \in V(G - S - v_1 - v_2 - v_3)$,
        \[
        f''(v) \,=\, \begin{cases}
                    f(v) - \deg_{\set{v_1,v_3}}(v) &\text{if } v \notin V(\bdry{G}),\\
                    f(v) &\text{if } v \in V(\bdry{G}) \setminus \set{v_1, v_2,v_3}.
                    \end{cases}
        \]
   Moreover, $x$ is the only vertex in $V(G)\setminus V(\partial G)$ adjacent to both $v_1$ and $v_3$, as otherwise there is a 2-chord violating Claim \ref{claim:2chord}.  Additionally, $x$ is not adjacent to any vertex in  $V(G-S-v_1-v_2-v_3) \cap V(\partial G)$ by Claim \ref{claim:no_other_neighbors}. Hence $(I \setminus \set{v_2}) \cup \{x\}$ is an independent set, and so applying Theorem \ref{thm:inductive} with $G-v_1-v_2-v_3$, $S$, and $(I \setminus \{v_2\}) \cup \{x\}$ in place of $G$, $S$, and $I$ shows that $G-S-v_1-v_2-v_3$ is weakly $f''$-degenerate. In other words, starting with $(G-S-v_1-v_2-v_3, f'')$, it is possible to remove every vertex via a sequence of legal applications of $\delsave$. Since $v_2$ has only one neighbor in $G-S-v_1-v_3$ (namely $x$), the same sequence of operations starting with $(G-S-v_1-v_3,f'')$ yields the pair $(H,g)$, where $V(H) = \{v_2\}$ and $g= (v_2 \mapsto 0)$.  It remains to apply the operation $\del(v_2)$ to remove the final vertex and conclude that $G-S-v_3-v_1$ is weakly $f''$-degenerate. But then $G-S$ is weakly $f$-degenerate, a contradiction.

\bigskip

\noindent
\textbf{Acknowledgments.} We are very grateful to Tao Wang for pointing out the error in the proof of \cite[Theorem 1.4]{WD} which prompted this work. We are also grateful to the anonymous referees for carefully reading the paper and providing helpful feedback.

%\bibliographystyle{siam}
%\bibliography{bibliog}
%\eve{The Yang reference is from \emph{Discrete} Math and the Voigt one, from \emph{Disc.} Math. The Kim, Kostochka, Li, Zhu one is from \emph{Discret.} Appl. Math. I dunno which abbreviation of discrete you prefer, but we should be consistent (and I'd rather note use \emph{discret.} if it's all the same to you).}
\frenchspacing

\printbibliography

\end{document}